\UseRawInputEncoding
\documentclass[12pt]{article}
\usepackage[centertags]{amsmath}
\usepackage{amsfonts}
\usepackage{amssymb}
\usepackage{latexsym}
\usepackage{amsthm}
\usepackage{newlfont}
\usepackage{graphicx}
\usepackage{listings}
\usepackage{booktabs}
\usepackage{abstract}
\usepackage{enumerate}
\usepackage{xcolor}
\RequirePackage{srcltx}
\date{}

\newlength{\defbaselineskip}
\newcommand{\setlinespacing}[1]%
           {\setlength{\baselineskip}{#1 \defbaselineskip}}

\newcommand{\actaqed}{\hfill $\actabox$}
{\medskip\noindent \textit{Proof of #1. }}%
{\actaqed \medskip}

\def\cA{{\mathcal A}}

\def\cK{{\mathcal K}}

\def\cT{{\mathcal T}}

\def\cV{{\mathcal V}}

\def\bbN{{\mathbb N}}

\def\bbR{{\mathbb R}}

\def\bbT{{\mathbb T}}

\def\bbZ{{\mathbb Z}}

\def\bF{{\mathbf F}}

\def\bH{{\mathbf H}}

\def\bJ{{\mathbf J}}
\def\bK{{\mathbf K}}

\def\bW{{\mathbf W}}

\def\ba{\mathbf a}

\def\bj{\mathbf j}
\def\bk{\mathbf k}

\def\bp{\mathbf p}
\def\bq{\mathbf q}
\def\br{\mathbf r}
\def\bs{\mathbf s}
\def\btt{\mathbf t}

\def\bx{\mathbf x}
\def\by{\mathbf y}

 \def \<{\langle}
\def\>{\rangle}

\def \e{\varepsilon}

\def \ff{\varphi}
\def\al{\alpha}
\def\bt{\beta}

\def\ga{\gamma}

\def\bt{\beta}

\newtheorem{Theorem}{Theorem}[section]
\newtheorem{Lemma}{Lemma}[section]
\newtheorem{Definition}{Definition}[section]
\newtheorem{Proposition}{Proposition}[section]
\newtheorem{Remark}{Remark}[section]

\numberwithin{equation}{section}

\newcommand{\be}{\begin{equation}}
\newcommand{\ee}{\end{equation}}

\begin{document}

\title{Entropy numbers of classes defined by integral operators}

\author{V. Temlyakov}

\newcommand{\Addresses}{{% additional braces for segregating \footnotesize
  \bigskip
  \footnotesize

%\medskip
%A.V. Gasnikov, \\ \textsc{Ivannikov institute for System Programming of Russian Academy of Sciences, Moscow, Russia;\\ Steklov Mathematical Institute of Russian Academy of Sciences, Moscow, Russia;  \\ Innopolis University, Tatarstan, Russia.
%\\ E-mail:} \texttt{gasnikov@yandex.ru}

 \medskip
  V.N. Temlyakov, \textsc{University of South Carolina, USA,\\ Steklov Mathematical Institute of Russian Academy of Sciences, Russia;\\ Lomonosov Moscow State University, Russia; \\ Moscow Center of Fundamental and Applied Mathematics, Russia.\\  
E-mail:} \texttt{temlyakovv@gmail.com}

}}
\maketitle

\begin{abstract}{In this paper we develop the following general approach. We study asymptotic behavior of the entropy numbers not for an individual smoothness class, how it is usually done, but for the collection of classes, which are defined by integral operators with kernels coming from a given class of functions. Earlier, such approach was realized for the Kolmogorov widths.}
\end{abstract}

\section{Introduction}
\label{In}

Some authors (see, for instance, the book \cite{Tikh}) list the following three important stages (periods) of 
development of classical approximation theory. 

(I) At the first stage, which was started by Chebyshev in 1854, approximation of individual functions 
by elements of a fixed set was studied. At that time the popular methods of approximation included approximation by 
algebraic and trigonometric polynomials and by rational functions. 

(II) It was understood (Weierstrass, de la Vall{\' e}e Poussin, Bernstein, Jackson) that the rate of best approximation by
the trigonometric polynomials of order $n$ with $n$ going to infinity is closely related to the smoothness properties of 
the function under approximation. This motivated the researchers to study asymptotic characteristics of approximation (for instance, best approximations by the trigonometric polynomials of order $n$) not for individual functions but for different smoothness classes. 

(III) Many different methods of approximation were used in approximation theory. Among them approximation by elements of finite-dimensional subspaces, approximation by elements of finite sets, approximation by rational functions, approximation by bilinear forms. Some other approximation methods became popular in numerical analysis, for instance,
approximation by piecewise linear functions. As a result the researchers started to formulate the optimization problems over 
collections of similar (in a certain sense) approximation methods. The concept of the Kolmogorov width (1936) and other 
related asymptotic characteristics were introduced. Thus, at this stage we begin with a given function class and a collection 
of approximation methods. For instance, a smoothness class and a collection of $n$-dimensional linear subspaces. Then, for a given smoothness class we try to find an optimal (in the sense of best approximation in a certain norm) linear subspace of dimension $n$. This leads to the Kolmogorov width. In the case of a collection of finite sets of fixed cardinality we end up with the concept of entropy numbers. 

In this paper we develop the following more general than in (III) approach, which was formulated in \cite{VT31}. 
A typical smoothness class can be defined with a help of an integral operator with a special kernel. For instance, in the case 
of periodic functions the Bernoulli kernel can be used. In \cite{VT31} we suggested to study asymptotic characteristics
(Kolmogorov widths) not for an individual smoothness class but for the collection of classes, which are defined by integral operators with kernels coming from a given class of functions. In a certain sense, this means that we have made in \cite{VT31} the step from stage (III) to a new stage and this step is similar to the step from stage (I) to stage (II). 

We now proceed to the detailed presentation. 

{\bf Problem formulation (general setting).} Let $\Omega^i$, $i=1,2$, be compact sets in $\bbR^d$ with probability measures $\mu_i$ on them.  Let $K(\bx,\by)$ be a measurable with respect to $\mu_1\times\mu_2$ function on $\Omega^1\times\Omega^2\subset \bbR^{2d}$.
Assume that for all $\ff\in L_q(\Omega^2,\mu_2)$   the integral
$$
J_K(\ff):=\int_{\Omega^2}K(\bx,\by)\varphi(\by)d\mu_2
$$
exists for almost all $\bx\in\Omega^1$. 

 For $1\le q\le \infty$   define the class
\be\label{In1}
\bW^K_q :=\left\{f:f = \int_{\Omega^2}K(\bx,\by)\varphi(\by)d\mu_2,\quad\|\varphi\|_{L_q(\Omega^2,\mu_2)}\le 1\right\}.  
\ee
Clearly, the class $\bW^K_q $ is the image of the unit ball of the space $L_q(\Omega^2,\mu_2)$ of the integral operator $J_K$. 

We are interested in the following problem. Assume that a class (a collection) $\bK$ of kernels $K$ is given. Consider a specific asymptotic characteristic $ac_n(\bW^K_q,L_p)$ of classes $\bW^K_q$ in the space $L_p(\Omega^1,\mu_1)$,
$1\le q,p\le \infty$. We want to estimate the following characteristics of the collection $\bK$
$$
ac_n(\bK,L_q,L_p) := \sup_{K\in\bK}  ac_n(\bW^K_q,L_p).
$$

\begin{Remark}\label{InR1} The above characteristic $ac_n(\bK,L_q,L_p)$ allows us to solve the following problem. Suppose that we are interested in an upper bound for the asymptotic characteristic $ac_n(\bW^K_q,L_p)$ and the only available  information about the kernel $K$ is the fact that it belongs to the collection $\bK$.  Then, clearly, the $ac_n(\bK,L_q,L_p)$ provides such an upper bound. 
\end{Remark} 

\begin{Remark}\label{InR2} We can formulate the above problem in a somewhat more general form in the functional analysis language. Let two Banach spaces $X$ and $Y$ be given. Denote $B_X:=\{f\in X\,:\, \|f\|_X\le 1\}$ the unit ball of the space $X$. Let $\bJ$ be a set of some linear bounded operators $J$ from $X$ to $Y$. For a specific asymptotic characteristic $ac_n(J(B_X),Y)$ of classes $J(B_X)$ in the space $Y$ consider the following characteristics of the set (collection) $\bJ$
$$
ac_n(\bJ,X,Y) := \sup_{J\in\bJ}  ac_n(J(B_X),Y).
$$
\end{Remark} 

{\bf Comment \ref{In}.1.} The classical classes $\bW^r_q$ defined in Section \ref{fc} is an example of classes $\bW^K_q $.
In this case $\Omega^i = \bbT^d$ is the $d$-dimensional cube $[0,2\pi]^d$,   and  $\mu_i = (2\pi)^{-d}d\bx$, $i=1,2$, is the normalized Lebesgue measure on $[0,2\pi]^d$. Then we have $\bW^r_q= \bW^{K_r}_q$ with $K_r(\bx,\by) := F_r(\bx-\by)$, where $F_r$ is the multivariate Bernoulli kernel defined in Section \ref{fc}. In the case $d=1$ and $r\in \bbN$ the class 
$\bW^r_q$ is directly related to the class of univariate periodic functions $f$ satisfying the condition $\|f^{(r)}\|_q \le 1$. 
In the case $d>1$ the class $\bW^r_q$ is directly related to the class of multivariate periodic functions with bounded mixed 
derivative. Therefore, the classes $\bW^r_q$ are smoothness classes with the smoothness parameter $r$. 
Classes $\bW^K_q$ appear naturally in other areas of mathematical research. We mention an example from the PDEs
(see, for instance, https://en.wikipedia.org/wiki/Green\%27s\_function and references therein). Consider a differential equation of the form
\be\label{In2}
L u(\bx) =f(\bx),\quad \bx \in \Omega
\ee
where $L$ is the linear differential operator. Let $G(\bx,\by)$ be a Green's function for the operator $L$. Then 
$$
u(\bx) := \int_\Omega G(\bx,\by)f(\by)d\by
$$
is a solution to the equation (\ref{In2}).

The above problem has two important ingredients -- the collection $\bK$ of kernels of integral operators and the asymptotic characteristic $ac_n(\cdot,\cdot,\cdot)$. The above problem is relatively well studied in the case, when the asymptotic characteristic $ac_n$ is the Kolmogorov width $d_n$. The reader can find the corresponding results in \cite{VT31}, \cite{VT32},  and \cite{VT47}. A brief description of some of those results can be found in \cite{VT211}. 
For illustration we formulate one of those results -- Theorem \ref{InT1}.      We need the following notation for $1\le q,p\le\infty$
 \be\label{ksi}
 \xi(q,p):= \left(\frac{1}{q} - \max\left(\frac{1}{2},\frac{1}{p}\right)\right)_+, \quad  (a)_+ :=\max(a,0).
\ee

 We also use the following convenient notations. We use $C$, $C'$ and $c$, $c'$ to denote various positive constants. Their arguments indicate the parameters, which they may depend on. Normally, these constants do not depend on a function $f$ and running parameters usually denoted by $m$, $n$, $k$. We use the following symbols for brevity. For two nonnegative sequences $a=\{a_n\}_{n=1}^\infty$ and $b=\{b_n\}_{n=1}^\infty$ the relation $a_n\ll b_n$ means that there is  a number $C(a,b)$ such that for all $n$ we have $a_n\le C(a,b)b_n$. Relation $a_n\gg b_n$ means that 
 $b_n\ll a_n$ and $a_n\asymp b_n$ means that $a_n\ll b_n$ and $a_n \gg b_n$. 
 For a real number $x$ denote $[x]$ the integer part of $x$.
  
 Here is a known result on the Kolmogorov widths. We present some discussion on the Kolmogorov widths of smoothness classes in Section \ref{D}. Let $X$ be a Banach space and $\bF\subset X$ be a  compact subset of $X$. The quantities  
$$
d_n (\bF, X) :=  \inf_{\{u_i\}_{i=1}^n\subset X}
\sup_{f\in \bF}
\inf_{c_i} \left \| f - \sum_{i=1}^{n}
c_i u_i \right\|_X, \quad n = 1, 2, \dots,
$$
are called the {\it Kolmogorov widths} of $\bF$ in $X$.

 \begin{Theorem}[{\cite{VT32}}]\label{InT1} Let $d=1$ and $\bF^\br_1$ denote one of the classes $\bW^\br_1$ or $\bH^\br_1$ (see the definition in Section \ref{fc} below) of functions of two variables. Then for $1\le q,p \le \infty$ and $\br > (1,1+\max(1/2,1/q))$   we have 
 $$
\sup_{K\in \bF^\br_1}  d_n(\bW^K_q)_p \asymp n^{-r_1-r_2 + \xi(q,p)} 
 $$
 with $\xi(q,p)$ defined in (\ref{ksi}).
 \end{Theorem}

% \begin{Theorem}[{\cite{VT47}}]\label{InT2} Let $d\in\bbN$ and $\mathbf{2} \le \bq \le \infty$, $2\le a <\infty$, $1<b<\infty$. 
%Assume that in the case $b\in (1,2]$ we have $r_i>0$, $i=1,2$, and in the case $b\in (2,\infty)$ we have $r_1>1/2$, $r_2>0$. Then
% $$
%\sup_{K\in \bW^\br_\bq}  d_m(\bW^K_a)_b \asymp (m^{-1}(\log m)^{d-1})^{r_1+r_2}m^{-1/2} .
 %$$
% \end{Theorem}

 In this paper we study the case, when the asymptotic characteristic is the entropy numbers and the collection $\bK$ is a class of multivariate periodic functions on $2d$ variables with mixed smoothness. We only discuss the cases $d=1$ and $d=2$ in this paper. We now proceed to the corresponding definitions. 

Let $X$ be a Banach space and let $B_X$ denote the unit ball of $X$ with the center at $0$. Denote by $B_X(y,r)$ a ball with center $y$ and radius $r$: $\{x\in X:\|x-y\|\le r\}$. For a compact set $A$ and a positive number $\e$ we define the covering number $N_\e(A)$:
$$
N_\e(A) := N_\e(A,X) :=\min \{n : \exists y^1,\dots,y^n :A\subseteq \cup_{j=1}^n B_X(y^j,\e)\}.
$$
and the entropy numbers $\e_k(A,X)$:
$$
\e_k(A,X) := \inf \{\e : \exists y^1,\dots ,y^{2^k} \in X : A \subseteq \cup_{j=1}
^{2^k} B_X(y^j,\e)\}.
$$
We refer the reader to \cite{VTbookMA} for detailed information on the entropy numbers of mixed smoothness classes. 

The main result of this paper is the following Theorem \ref{MT1}, which is proved in Section \ref{M}. We present some discussion on the entropy numbers of smoothness classes in Section \ref{D}.

 \begin{Theorem}\label{MT1} Let   $r>1$  be given. Then, we have (classes $\bH^{r,2d}_\bq$ are defined in Section \ref{fc})
 \be\label{Md1}
 \sup_{K\in \bH^{r,2}_{1,\infty}} \e_k(\bW^K_1,L_\infty) \ll k^{-2r} (\log 2k)^{2r};
 \ee
  \be\label{Md2}
 \sup_{K\in \bH^{r,4}_{1,\infty}} \e_k(\bW^K_1,L_\infty) \ll k^{-2r} (\log 2k)^{4r+5/2}.
 \ee
 \end{Theorem}

 We note that similar results with somewhat worse logarithmic factors can be also proved for $d>2$. We do not present these results here because we are still working on improving those logarithmic factors. We note that 
 the asymptotic behavior of the entropy numbers $\e_k(\bW^r_2,L_\infty)$ is closely related to the famous probability 
 problem on the Small Ball Inequality, which is solved in the case $d=2$ (see \cite{Tal}) and still open in the case $d>2$. The reader can find a detailed discussion of history of results on $\e_k(\bW^r_q,L_p)$ in \cite{VTbookMA},
 pp.383--386. Further discussion and some open problems are given in Section \ref{D}.

\section{Function classes}
\label{fc}

First of all we define the vector $L_{\bq}$-norm, $\bq=(q_1,\dots,q_v)$, of functions of $v$ variables $\bx=(x_1,\dots,x_v)$ as
$$
\|f(\bx)\|_{\bq} := \|\cdots\|f(\cdot,x_2,\dots,x_v)\|_{q_1}\cdots\|_{q_v}.
$$

We begin with the definition of classes $\bW^\ba_\bq$ (see, for instance, \cite{VTmon}, p.31, in the case of scalar $q$).
\begin{Definition}\label{fcD1}
In the univariate case, for $a>0$, let
\be\label{Bi8}
F_a(x):= 1+2\sum_{k=1}^\infty k^{-a}\cos (kx-a\pi/2)
\ee
be the Bernoulli kernel and in the multivariate case, for $\ba=(a_1,\dots,a_v) \in \bbR^v_+$, $\bx=(x_1,\dots,x_v)\in \bbT^v$, let
\be\label{Bi8m}
F_\ba(\bx) := \prod_{j=1}^v F_{a_j}(x_j).
\ee
Denote for $\mathbf{1}\le \bq\le \infty$ (we understand the vector inequality coordinate wise)
$$
\bW^\ba_\bq := \{f:f=\varphi\ast F_\ba,\quad \|\varphi\|_\bq \le 1\},
$$
where
$$
( F_\ba \ast \varphi)(\bx):= (2\pi)^{-v}\int_{\bbT^v} F_\ba(\bx-\by) \ff(\by)d\by.
$$
\end{Definition}
The classes $\bW^\ba_\bq$ are classical classes of functions with {\it dominated mixed derivative}
(Sobolev-type classes of functions with mixed smoothness).
 
We now proceed to the definition of the classes $\bH^\ba_\bq := \bH^{\ba,v}_\bq$ of periodic functions of $v$ variables, which is based on the mixed differences (see, for instance, \cite{VTmon}, p.31, in the case of scalar $q$).  
 
\begin{Definition}\label{fcD2}
Let  $\btt =(t_1,\dots,t_v)$ and $\Delta_{\btt}^l f(\bx)$
be the mixed $l$-th difference with
step $t_j$ in the variable $x_j$, that is
$$
\Delta_{\btt}^l f(\bx) :=\Delta_{t_v,v}^l\cdots\Delta_{t_1,1}^l
f(x_1,\dots ,x_v) .
$$
Let $e$ be a subset of natural numbers in $[1,v]$. We denote
$$
\Delta_{\btt}^l (e) :=\prod_{j\in e}\Delta_{t_j,j}^l,\qquad
\Delta_{\btt}^l (\varnothing) := Id \,-\, \text{identity operator}.
$$
We define the class $\bH_{\bq,l}^\ba B$, $l > \|\ba\|_\infty$,   as the set of
$f\in L_\bq(\bbT^v)$ such that for any $e$
\be\label{Bi9}
\bigl\|\Delta_{\btt}^l(e)f(\bx)\bigr\|_\bq\le B
\prod_{j\in e} |t_j |^{a_j} .
\ee
In the case $B=1$ we omit it. It is known (see Theorem \ref{H} below)  that the classes $\bH^\ba_{\bq,l}$ with different $l>\|\ba\|_\infty$ are equivalent. So, for convenience we omit $l$ from the notation. 
\end{Definition}

We now formulate a result, which gives an equivalent description of classes $\bH^\ba_{\bq,l}$. 
 We need some classical trigonometric polynomials. The univariate Fej\'er kernel of order $j - 1$:
$$
\mathcal K_{j} (x) := \sum_{|k|\le j} \bigl(1 - |k|/j\bigr) e^{ikx} 
=\frac{(\sin (jx/2))^2}{j (\sin (x/2))^2}.
$$
The Fej\'er kernel is an even nonnegative trigonometric
polynomial of order $j-1$.  It satisfies the obvious relations
\be\label{FKm}
\| \mathcal K_{j} \|_1 = 1, \qquad \| \mathcal K_{j} \|_{\infty} = j.
\ee
Let $\cK_\bj(\bx):= \prod_{i=1}^v \cK_{j_i}(x_i)$ be the $v$-variate Fej\'er kernels for $\bj = (j_1,\dots,j_d)$ and $\bx=(x_1,\dots,x_v)$.

The univariate de la Vall\'ee Poussin kernels are defined as follows
$$
\cV_m := 2\cK_{2m} - \cK_m.
$$
We also need the following special trigonometric polynomials.
Let $s$ be a nonnegative integer. We define
$$
\mathcal A_0 (x) := 1, \quad \mathcal A_1 (x) := \mathcal V_1 (x) - 1, \quad
\mathcal A_s (x) := \mathcal V_{2^{s-1}} (x) -\mathcal V_{2^{s-2}} (x),
\quad s\ge 2,
$$
where $\mathcal V_m$ are the de la Vall\'ee Poussin kernels defined above.
For $\bs=(s_1,\dots,s_v)\in \bbN^v_0$ define
$$
\cA_\bs(\bx) := \prod_{j=1}^v  \cA_{s_j}(x_j),\qquad \bx=(x_1,\dots,x_v)
$$
and
$$
A_\bs(f) := \cA_\bs \ast f.
$$

The following result is known (see, for instance, \cite{VTmon}, p.32, for the scalar $q$ and \cite{VT32} for the vector $\bq$).

\begin{Theorem}\label{H} Let $f\in \bH^\ba_{\bq,l}$, $\mathbf 1 \le \bq \le \infty$. Then, for $\bs \ge \mathbf 0$
\be\label{H1}
\|A_\bs(f)\|_\bq \le C(\ba,v,l)2^{-(\ba,\bs)}.
\ee
Conversely, from (\ref{H1}) it follows that there exists $B>0$, which does not depend on $f$, such that $f\in \bH^\ba_{\bq,l}B$.
\end{Theorem}

The reader can find results on approximation properties of these classes in the books \cite{VTmon}, \cite{VTbookMA}, and \cite{DTU}.
In this paper we consider the case, when $v=2d$, $d\in \bbN$, $\mathbf{1}\le \bq \le \infty$, and $\ba$ has a special form: $a_j = r$ for $j=1,\dots,2d$. In this case we write $\bW^r_\bq$ and $\bH^{r,2d}_\bq$.  Typically, $\bq$ has the form $q_j=h_1$, $q_{j+d}=h_2$, $j=1,2,\dots,d$. In such case we write, for instance, $\bH^{r,2d}_{h_1,h_2}$. 

{\bf Some general classes.} Let $\Omega^i$, $i=1,2$, be compact sets in $\bbR^d$ with probability measures $\mu_i$ on them.  

  Clearly, we have $\bW^r_q = \bW^{K}_q $ with $\Omega^i = \bbT^d$ and $\mu_i$, $i=1,2$, being the normalized on $\bbT^d$ Lebesgue measure and $K(\bx,\by) := F_r(\bx-\by)$.

 We begin with the general definition of classes of our interest (see \cite{VTbookMA}, p.371). Let $X$   be a Banach space.
Suppose a sequence of finite dimensional subspaces $X_n \subset X$, $n=1,\dots $, is given. Define for $a\in \bbR_+$, $b\in \bbR$, the following class 
$$
{\bar \bW}^{a,b}_X:={\bar \bW}^{a,b}_X\{X_n\} := \Big\{f\in X: f=\sum_{n=1}^\infty f_n,\quad  f_n\in X_n, 
$$
$$
 \|f_n\|_X \le 2^{-an}n^{b},\quad n=1,2,\dots\Big\}.
$$

We now formulate a known result on the entropy numbers of classes ${\bar \bW}^{a,b}_X$ (see \cite{VTbookMA}, p.371).
Let $Y\subset X$ be a Banach space.  
 Denote $D_n:=\dim X_n$ and assume that for the unit balls $B(X_n):=\{f\in X_n: \|f\|_X\le 1\}$ we have the following upper bounds for the entropy numbers: there exist real $\al$ and nonnegative   $\ga$ and $\bt\in(0,1]$ such that 
\be\label{EA}
\e_k(B(X_n),Y) \ll n^\al \left\{\begin{array}{ll} (D_n/(k+1))^\bt (\log (4D_n/(k+1)))^\ga, &\quad k\le 2D_n,\\
2^{-k/(2D_n)},&\quad k\ge 2D_n.\end{array} \right.
\ee

\begin{Theorem}\label{fcT1} Assume $D_n \asymp 2^n n^c$, $c\ge 0$, $a>\bt$, and subspaces $\{X_n\}$ satisfy (\ref{EA}). Then
\be\label{fc4}
\e_k(\bar \bW^{a,b}_X\{X_n\},Y) \ll k^{-a} (\log k)^{ac+b+\al}.
\ee
\end{Theorem}

We need a version of Theorem \ref{fcT1}, which we prove below. Let $n_0\in \bbN$. Define the following analog of the class ${\bar \bW}^{a,b}_X$
$$
\bW:={\bar \bW}^{a,b,b'}_{X,n_0}:={\bar \bW}^{a,b,b'}_{X,n_0}\{X_n\} := \Big\{f\in X: f=\sum_{n=n_0+1}^\infty f_n,\quad  f_n\in X_n, 
$$
$$
 \|f_n\|_X \le 2^{-an}n^{b}(n-n_0)^{b'},\quad n=n_0+1,n_0+2,\dots\Big\}.
$$
In our applications we always have $X_1 \subseteq X_2 \cdots \subseteq X_n \subseteq \cdots$ and, therefore, 
\be\label{Dn}
D_n\le D_{n+1},\quad n=1,2,   \dots. 
\ee
 
 \begin{Theorem}\label{fcT2} Let $n_0:=2u+1$, $u\in \bbN$. Assume that $\{D_n\}$ satisfies (\ref{Dn}) and $D_n \asymp 2^{n-u} n^c$, $n\ge 2u$, $c\ge 0$. Suppose that subspaces $\{X_n\}$ satisfy (\ref{EA}) for $n>n_0$. Then for $a>\bt$ there exists a constant $C=C(a,\beta)$    such that
\be\label{fc5}
\e_{Ck}(\bar \bW^{a,b,b'}_{X,n_0}\{X_n\},Y) \ll k^{-2a} (\log k)^{2ac+b+\al},\quad k := D_{2u},
\ee
with a constant in $\ll$ independent of $u$. 
\end{Theorem}
\begin{proof}   For $l\ge n_0+1$ we define $k_l := [D_{2u}2^{\mu(n_0-l)}]$, $\mu:=(a-\bt)/(2\bt)$. Then 
$$
\sum_{l\ge n_0+1} k_l \le Ck,\quad C=C(a,\bt).
$$
It follows from the definition of the class $\bW:=\bar \bW^{a,b,b'}_{X,n_0}$ that  
\be\label{fc5a}
\e_{Ck}( \bW,Y) \le \sum_{l=n_0+1}^\infty 2^{-al}l^{b}(l-n_0)^{b'}\e_{k_l}(B(X_l),Y).
\ee
The above definition of $k_l$ and relation (\ref{Dn}) imply that for $l>n_0$ we have $k_l \le D_l$.     By our assumption (\ref{EA}) we get  
$$
 \e_{k_l}(B(X_l),Y) \ll l^\al \left(\frac{D_l}{k_l+1}\right)^\bt  \left(\log\frac{4D_l}{k_l+1}\right)^\ga 
 $$
 $$
 \le l^\al  \left(\frac{D_l}{D_{2u}}\right)^\bt 2^{\bt\mu(l-n_0)} \left(\log(4D_l/D_{2u}) + \bt\mu(l-n_0) \right)^\ga.
$$
Substituting this bound in (\ref{fc5a}), we obtain
$$
\e_{Ck}( \bW,Y)\ll \sum_{l\ge n_0+1} 2^{-al} l^{b+\al}2^{\mu(l-n_0)\bt} (D_l/D_{2u})^\bt (l-n_0)^\ga (l-n_0)^{b'}
\ll 2^{-an_0}n_0^{b+\al} .
$$
Taking into account that $k\asymp 2^{n_0/2}n_0^c$, we conclude that
\be\label{fc6}
\e_{Ck}(\bW,Y) \ll k^{-2a}(\log k)^{2ac+b+\al}.
\ee
%Taking into account that the right hand side in (\ref{fc6}) decays polynomially, we conclude that the upper bound in (\ref{fc5}) holds. 

\end{proof}

\section{Main results}
\label{M}

We begin with the following Lemma \ref{ML1}. We need some more notation. Let $Q_n$, $n\in \bbN$, be the stepped hyperbolic cross:
$$
Q_n := \bigcup_{\bs:\|\bs\|_1\le n} \rho(\bs),
$$
where
$$
\rho (\bs) := \{\bk \in \bbZ^d : [2^{s_j-1}] \le |k_j| < 2^{s_j}, \quad j=1,\dots,d\},
$$
and let the corresponding set of the hyperbolic cross polynomials be $\cT(Q_n)$. Also, for a finite subset 
$Q\subset \bbZ^d$ we denote a subspace of the trigonometric polynomials with frequencies in $Q$ by
$$
\cT(Q) := \left\{f: f=\sum_{\bk\in Q} c_\bk e^{i(\bk,\bx)}\right\}.
$$

\begin{Lemma}\label{ML1} Let $d\in \bbN$ and $r>1$  be given. For any $K\in \bH^{r,2d}_{1,\infty}$ and any 
$u \in\bbN$ there exists a subspace $S_{u} \subset L_\infty(\bbT^d)$, $\dim S_{u} \le |Q_{u}|$, such that each $f\in \bW^K_1$ has a representation
\be\label{M1}
f=f^1_{u} + f^2_{u}+f^3_{u}
\ee
with the following properties. We have that $f^1_{u} \in \cT(Q_{u})$, $f^2_{u}\in S_{u}$ and satisfy the bounds $\|f^i_{u}\|_\infty \le C(r,d)$, $i=1,2$. Also, $f^3_{u}$ belongs to the class ${\bar \bW}^{a,b,b'}_{X,n_0}\{X_n\}$ with 
$X=L_1(\bbT^d)$, $n_0=2u+1$, $X_n := \cT(Q_n)$, $a=r$, $b=2d-2$, and $b'=1$.
\end{Lemma}
\begin{proof} We write $\bs\in \bbN_0^{2d}$ in the form $\bs = (\bs^1,\bs^2)$, $\bs^i \in \bbN_0^d$, $i=1,2$. 
Then 
$$
K=\sum_\bs A_\bs(K) = \sum_{\bs^1} \sum_{\bs^2} A_{(\bs^1,\bs^2)}(K).
$$
For a given $u \in \bbN$ define the following kernels.
$$
K^1_{u}:=\sum_{\bs^1:\|\bs^1\|_1\le u} \sum_{\bs^2} A_{(\bs^1,\bs^2)}(K),
$$
$$
K^2_{u}:=\sum_{\bs^2:\|\bs^2\|_1\le u} \sum_{\bs^1} A_{(\bs^1,\bs^2)}(K),
$$
$$
K^3_{u}:=\sum_{\bs^1:\|\bs^1\|_1> u} \sum_{\bs^2:\|\bs^2\|_1> u} A_{(\bs^1,\bs^2)}(K).
$$
We now establish some simple properties of these kernels. Theorem \ref{H} guarantees that 
\be\label{M2}
\|A_\bs(K)\|_{1,\infty} \ll 2^{-r\|\bs\|_1}.
\ee
By the Nikol'skii inequality (see, for instance, \cite{VTbookMA}, p.90) relation (\ref{M2}) implies
\be\label{M3}
\|A_\bs(K)\|_{\infty} \ll 2^{-(r-1)\|\bs\|_1}.
\ee
Therefore, we have 
\be\label{M4}
\|K^i_{u}\|_\infty \ll 1,\quad i=1,2.
\ee
Let
$$
f  = (2\pi)^{-d}\int_{\bbT^d} K(\bx,\by)\ff(\by)d\by.
$$
Define
$$
f^i_{u} := (2\pi)^{-d}\int_{\bbT^d} K^i_{u}(\bx,\by)\ff(\by)d\by, \quad i=1,2,3.
$$
Then, the required properties of $f^1_{u}$ and $f^2_{u}$ follow from their definitions and the inequailities 
(\ref{M4}). 

We now proceed to the $f^3_{u}$. We set $n_0:=2u+1$ and for $n>n_0$ define the kernels
$$
K_{n,u} := \sum_{\bs^1:\|\bs^1\|_1> u; \,\bs^2:\|\bs^2\|_1> u;\, \|\bs^1\|_1+\|\bs^2\|_1=n} A_{(\bs^1,\bs^2)}(K).
$$
Then inequality (\ref{M2}) implies
\be\label{M5}
\|K_{n,u}\|_{1,\infty} \ll 2^{-rn}n^{2d-2}(n-2u-1).
\ee
Define
$$
f_{n,u} := (2\pi)^{-d}\int_{\bbT^d} K_{n,u}(\bx,\by)\ff(\by)d\by.
$$
Then
\be\label{M6}
\|f_{n,u}\|_1 \le (2\pi)^{-d}\int_{\bbT^d} \|K_{n,u}(\cdot,\by)\|_1|\ff(\by)|d\by \le \|K_{n,u}\|_{1,\infty}.
\ee
Clearly, $f_{n,u} \in \cT(Q_{n-u})$. Relations (\ref{M5}) and (\ref{M6}) together with the obvious representation
$$
f^3_{u} = \sum_{n=n_0+1}^\infty f_{n,u}
$$
prove that $f^3_{u}\in {\bar \bW}^{a,b,b'}_{X,n_0}\{X_n\}$.

\end{proof}

We now proceed to the proof of the main result of the paper -- Theorem \ref{MT1}. 

 {\bf Proof of Theorem \ref{MT1}.} The proof is based on Lemma \ref{ML1}. For a given $k\in \bbN$ we find $u$ such that 
 $k \asymp 2^uu^d$ and then apply Lemma \ref{ML1} and results on the entropy numbers including Theorem \ref{fcT2}. Later, we explain in more detail how to choose parameter $u$. By Lemma \ref{ML1} we have 
 the representation (\ref{M1}). We begin with the $f^i_u$, $i=1,2$. These functions belong to the balls $B^i$, $i=1,2$, of radius $C(r,d)$ of finite-dimensional subspaces of the Banach space $L_\infty(\bbT^d)$. We use the following well known result (see, for instance, \cite{VTbookMA}, p.324).
 
 \begin{Proposition}\label{MP1} For any $n$-dimensional Banach space $X$ we have
$$
\e^{-n} \le N_\e(B_X,X) \le (1+2/\e)^n,
$$
and, therefore,
$$
\e_k(B_X,X) \le 3(2^{-k/n}).
$$
\end{Proposition}

In our case $n=|Q_u| \asymp 2^{u}u^{d-1}$. Therefore, by taking (with large enough constants) $k \asymp 2^uu^d \asymp n \log n$ we can make $\e_k(B^i,L_\infty) \ll k^{-c}$ with arbitrarily large $c$. 

We now proceed to the $f^3_u$. By Lemma \ref{ML1} we have that $f^3_{u}$ belongs to the class $\bW:={\bar \bW}^{a,b,b'}_{X,n_0}\{X_n\}$ with 
$X=L_1(\bbT^d)$, $n_0=2u+1$, $X_n := \cT(Q_{n-u})$, $a=r$, $b=2d-2$, and $b'=1$.
We want to apply Theorem \ref{fcT2}. For that we need to check the condition (\ref{EA}). 
Let us discuss the cases $d=1$ and $d=2$ separately.

{\bf The case $d=1$.} The following result holds for $d=1$ (see, for instance, \cite{VTbookMA}, p.345). 
 
\begin{Proposition}\label{MP2} We have for $d=1$
$$
\e_k(\cT( Q_n)_1,L_\infty) \ll  \left\{\begin{array}{ll} (|Q_n|/k) \log (4|Q_n|/k), &\quad k\le 2| Q_n|,\\
2^{-k/(2| Q_n|)},&\quad k\ge 2| Q_n|.\end{array} \right.
$$
Note that in this case $Q_n = (-2^n,2^n)$.
\end{Proposition}

Proposition \ref{MP2} implies that condition (\ref{EA}) is satisfied with $\alpha =0$, $\bt =1$, and $\gamma =1$.
In this case $c=0$ and $D_n\asymp 2^{n-u}$. We now apply Theorem \ref{fcT2} with $k' := |Q_u| \asymp k /\log k$, $a=r$, $b=0$, $b'=1$. We obtain
$$
\e_k(\bW,L_\infty) \le \e_{k'}(\bW,L_\infty) \ll (k')^{-2r} \ll k^{-2r} (\log k)^{2r}.
$$

{\bf The case $d=2$.} The following result holds for $d=2$ (see, for instance, \cite{VTbookMA}, p.363). 
 
\begin{Proposition}\label{MP3} We have for $d=2$
$$
\e_k(\cT( Q_n)_1,L_\infty) \ll  \left\{\begin{array}{ll} n^{1/2}(| Q_n|/k) \log (4| Q_n|/k), &\quad k\le 2| Q_n|,\\
n^{1/2}2^{-k/(2| Q_n|)},&\quad k\ge 2| Q_n|.\end{array} \right.
$$
\end{Proposition}

Proposition \ref{MP3} implies that condition (\ref{EA}) is satisfied with $\alpha =1/2$, $\bt =1$, and $\gamma =1$.
In this case $c=1$ and $D_n \asymp 2^{n-u}n$. We now apply Theorem \ref{fcT2} with $k' := |Q_u| \asymp k /\log k$, $a=r$, $b=2$, $b'=1$. We obtain
$$
\e_k(\bW,L_\infty) \le \e_{k'}(\bW,L_\infty) \ll (k')^{-2r}(\log (k'))^{2r+5/2} \ll k^{-2r} (\log k)^{4r+5/2}.
$$

This completes the proof of Theorem \ref{MT1}.

Let us now discuss the case of $Y=L_p(\bbT^2)$ with $p<\infty$. The following result holds for $d=2$ (see, for instance, \cite{VTbookMA}, p.361).

\begin{Proposition}\label{MP4} Let $ 1< p <\infty$ and $\bt:=1-1/p$. Then for $d=2$
$$
\e_k(\cT(Q_n)_1,L_p) \ll  \left\{\begin{array}{ll}  (|Q_n|/k)^\bt (\log (4|Q_n|/k))^\bt, &\quad k\le 2|Q_n|,\\
 2^{-k/(2|Q_n|)},&\quad k\ge 2|Q_n|.\end{array} \right.
$$
\end{Proposition}

Proposition \ref{MP4} implies that condition (\ref{EA}) is satisfied with $\alpha =0$, $\bt =1-1/p$, and $\gamma =1-1/p$. In this case $c=1$ and $D_n \asymp 2^{n-u}n$. We now apply Theorem \ref{fcT2} with $k' := |Q_u| \asymp k /\log k$, $a=r$, $b=2$, $b'=1$. We obtain
$$
\e_k(\bW,L_p) \le \e_{k'}(\bW,L_p) \ll (k')^{-2r}(\log (k'))^{2r+2} \ll k^{-2r} (\log k)^{4r+2}.
$$
We formulate it as a theorem.

 \begin{Theorem}\label{MT2} Let  $r>1$  be given. Then, we have for $p<\infty$
  $$
 \sup_{K\in \bH^{r,4}_{1,\infty}} \e_k(\bW^K_1,L_p) \ll k^{-2r} (\log 2k)^{4r+2}.
 $$
 \end{Theorem}

 \section{Some lower bounds}
 \label{L}
 
 Consider the function
 $$
 K(\bx,\by) := F_r(\bx-\by), \quad \bx,\by \in \bbT^d, \quad d\in\bbN,
 $$
 where $F_r$ is the multivariate Bernoulli kernel defined in (\ref{Bi8m}) with $\ba =(r,\dots,r)$ and $v=d$. 
 
 \begin{Lemma}\label{LL1} We have $F_{2r}(\bx-\by) \in \bH^{r,2d}_{1,\infty}C(r,d)$.
  \end{Lemma}
  \begin{proof} Let $\bs = (\bs^1,\bs^2)$ with $\bs^i \in \bbN_0^d$, $i=1,2$. Assume that $\|\bs^1\|_1 \ge \|\bs^2\|_1$.
  The other case $\|\bs^1\|_1 \le \|\bs^2\|_1$ is analyzed in the same way. It is easy to check that
  $$
  \|A_\bs(F_{2r}(\bx-\by))\|_{1,\infty} \le C(d)\|A_{\bs^1}(F_{2r})(\bx-\by)\|_{1,\infty} = C(d)\|A_{\bs^1}(F_{2r})\|_1.
  $$
  Using the well known results on the Bernoulli kernel $F_{2r}$ (see, for instance, \cite{VTbookMA}, p.164), we continue 
  $$
  \le C(r,d)2^{-2r\|\bs^1\|_1} \le C(r,d)2^{-r\|\bs\|_1}.
  $$
  This completes the proof. 
  \end{proof}
  
  The following result is known (see, for instance, \cite{VTbookMA}, p.363, Theorem 7.7.2). 
  
 \begin{Theorem}\label{Wlo} For $r>0$ and $1\le q \le\infty$ one has
$$
\e_k(\bW^r_q,L_1) \gg k^{-r}(\log k)^{r (d-1)}.
$$
\end{Theorem}

Note that because of the monotonicity of the $L_q$ norms we always have for $1 \le q,p \le \infty$
\be\label{L1}
\e_k(\bW^K_\infty,L_1) \le \e_k(\bW^K_q,L_p) \le \e_k(\bW^K_1,L_\infty).
\ee

Thus Theorems \ref{MT1} and \ref{Wlo}, combined with Lemma \ref{LL1} give the following bounds for $1\le q,p\le\infty$. Note, that the upper bounds hold under condition $r>1$. 

In the case $d=1$ 
\be\label{L2}
 k^{-2r} \ll \sup_{K\in \bH^{r,2}_{1,\infty}} \e_k(\bW^K_q,L_p) \ll k^{-2r} (\log 2k)^{2r};
 \ee
 and in the case $d=2$
 \be\label{L3}
 k^{-2r}(\log k)^{2r} \ll \sup_{K\in \bH^{r,4}_{1,\infty}} \e_k(\bW^K_q,L_p) \ll k^{-2r} (\log 2k)^{4r+5/2}.
 \ee
 In Section \ref{D} somewhat stronger lower bounds than in (\ref{L3}) are given for specific 
 pairs $q$ and $p$ (see (\ref{LB1}) -- (\ref{LB4})). 
 
 \section{Discussion}
\label{D}

We begin with a brief historical comments on the Kolmogorov widths   of the univariate classes 
$\bW^r_q$. We use the notation $W^r_q$ for these classes in order to emphasize that they are classes of the univariate functions. A more detailed history the reader can find, for instance, in \cite{VTbookMA}, p.79, Section 2.5.  
The Kolmogorov width was introduced in the Kolmogorov's
paper \cite{Ko36}.   The following  theorem is known (see, for instance, \cite{VTbookMA}, p.37).

\begin{Theorem}\label{aT4.1} Let $1\le q$,  $p\le \infty$,  $r > r(q, p)$,  then
$$
d_n (W_{q}^r, L_p)\asymp   n^{-r+\left(1/q-\max(1/2, 1/p)\right)_+},
$$
where
$$
r(q, p) := \begin{cases} (1/q-1/p)_+\qquad&\text{ for }1\le q\le p \le2; \qquad
1\le p\le q\le\infty, \\
\max(1/2, 1/q)\qquad&\text{ otherwise }.\end{cases}
$$
\end{Theorem}

The results connected with Theorem \ref{aT4.1}
were obtained by a number of authors. Exact values of widths
have been obtained in some cases but we shall not discuss these results here.
The first result was obtained by Kolmogorov \cite{Ko36}
($q=p=2$, in this case the exact values of the widths were obtained).
Rudin \cite{Ru2} investigated the case $r=1$, $q=1$, $p=2$.
Stechkin \cite{St} generalized the Rudin's result to all $r$ and
investigated the case $q=p=\infty$. For $1\le q=p<\infty$ the
orders of the widths were obtained by Babadzhanov and
Tikhomirov~\cite{BaT}. Makovoz \cite{Mak} investigated the case
$1\le p< q\le\infty$. Ismagilov \cite{I} found the orders of the
widths for $1\le q<p\le 2$ and proved the estimate
$$
d_m(W_1^2,L_{\infty})\le C(\varepsilon)m^{-6/5+\varepsilon},
\qquad \varepsilon>0.
$$
In particular, this estimate shows that in the case $q=1$,
$p=\infty$ the subspace $\cT(n)$ of trigonometric polynomials is not
 optimal from the point of view of the Kolmogorov widths.
This result was the first of such kind. The orders of the widths in
the case $q=1$, $p>2$, $r\ge 2$ were obtained by Gluskin \cite{Gl1}.
For $1<q<p$, $p>2$ the orders of the widths were obtained by
Kashin \cite{Ka}, who developed a new breakthrough technique.

In the case $d>1$ the following Theorem \ref{Ts7.1} is known (see, for instance, \cite{VTbookMA}, p.216). A more detailed discussion of the corresponding results the reader can find, for instance, in \cite{VTbookMA}, Ch.5  and in \cite{DTU}.  

\begin{Theorem}\label{Ts7.1} Let $r(q,p)$ be the same as in Theorem \ref{aT4.1} above.
Then
$$
d_n(\bW_q^r,L_p) \asymp
\begin{cases}
\left( \frac{(\log n)^{d-1}}{n}\right)^{r-\bigl(1/q-\max(1/2,1/p)\bigr)_+}
\text{ for}&1<q,p<\infty,\\
& r>r(q,p)\\
\left( \frac{(\log n)^{d-1}}{n}\right)^{r-1/2}(\log m)^{(d-1)/2}
\text{ for}&q=1,\\
&2\le p<\infty,\ r>1.
\end{cases}
$$
\end{Theorem}

Theorem \ref{Ts7.1} for $q=p=2$ was proved by Babenko \cite{Ba1}, for
$1<q=p<\infty$ by Mityagin \cite{Mi} ($r$ natural) and by
Galleev \cite{Ga1} ($r$ arbitrary); for $1<q<p\le 2$ by
Temlyakov \cite{Tem2}, \cite{Tem6}; for $1\le q<p<\infty$, $2\le p<\infty$
by Temlyakov \cite{Tem5}, \cite{Tem10}; and for $1<p<q<\infty$ by Galeev \cite{Ga3}.

We now give some comments on the known results about entropy numbers. The reader can find further discussion 
in the books \cite{LGM}, Ch.15, \cite{VTbookMA}, Ch.7, and \cite{VTbook}, Ch.3 and in the papers \cite{Schu}, \cite{TE3}, \cite{VT155}.  

The concept of entropy was introduced by Kolmogorov in \cite{Ko56}. It is well known (see \cite{KoTi} and \cite{BS}) that in the univariate case ($d=1$)  
\begin{equation}\label{36.1}
\e_k(W^r_{q},L_p)\asymp k^{-r}
\end{equation}
 holds for all $1\le q,p \le \infty$ and $r>(1/q-1/p)_+$. 
 We note that condition $r>(1/q-1/p)_+$ is a necessary and sufficient condition for compact embedding of $\bW^r_q$ into $L_p$. Thus, (\ref{36.1}) provides a complete description of the rate of $\e_k(W^r_q,L_p)$ in the univariate case. We point out that (\ref{36.1}) shows that the rate of decay of $\e_k(W^r_q,L_p)$ depends only on $r$ and does not depend on $q$ and $p$. In this sense the strongest upper bound (for $r>1$) is $\e_k(W^r_{1},L_\infty) \ll k^{-r}$ and the strongest lower bound is $\e_k(W^r_{\infty},L_1)\gg k^{-r}$. 

There are different generalizations of classes $W^r_q$ to the case of multivariate functions. In this section we only discuss known results for classes $\bW^r_q$ of functions with bounded mixed derivative.  
 
The problem of estimating $\e_k(\bW^r_q,L_p)$ has a long history. The first result on the right order of $\e_k(\bW^r_{2},L_2)$ was obtained by Smolyak \cite{Smo}. Later (see \cite{TE1}, \cite{TE2} and \cite{VTbookMA}, Ch.7) it was established that
\begin{equation}\label{36.2}
\e_k(\bW^r_q,L_p)\asymp k^{-r}(\log k)^{r(d-1)} 
\end{equation}
holds for all $1<q,p<\infty$, $r>1$. The case $1<q=p<\infty$, $r>0$ was established by Dinh Dung \cite{DD}. Belinskii \cite{Bel} extended (\ref{36.2}) to the case $r>(1/q-1/p)_+$ when $1<q,p<\infty$. 

It is known in approximation theory (see \cite{VTbookMA}) that investigation of asymptotic characteristics of classes $\bW^r_q$ in $L_p$ becomes more difficult when $q$ or $p$ takes value $1$ or $\infty$ than when $1<q,p<\infty$. It turns out to be the case for $\e_k(\bW^r_q,L_p)$ as well. It was discovered that in some of these extreme cases ($q$ or $p$ equals $1$ or $\infty$) relation (\ref{36.2}) holds and in other cases it does not hold. We describe the picture in detail. It was proved in \cite{TE2} that (\ref{36.2}) holds for $p=1$, $1<q<\infty$, $r>0$. It was also proved that (\ref{36.2}) holds for $p=1$, $q=\infty$ (see \cite{Bel} for $r>1/2$ and  \cite{KTE2} for $r>0$). Summarizing, we state that (\ref{36.2}) holds for $1<q,p<\infty$ and $p=1$, $1<q\le\infty$ for all $d$ (with appropriate  restrictions on $r$). This easily implies that (\ref{36.2}) also holds for $q=\infty$, $1\le p<\infty$. For all other pairs $(q,p)$, namely, for $p=\infty$, $1\le q\le\infty$ and $q=1$, $1\le p\le \infty$ the rate of $\e_k(\bW^r_q,L_p)$ is not known in the case $d>2$. It is an outstanding open problem. 

In the case $d=2$ this problem is essentially solved. We now cite the corresponding results. The first result on the right order of $\e_k(\bW^r_q,L_p)$ in the case $p=\infty$ was obtained by Kuelbs and Li \cite{KL} for $q=2$, $r=1$. It was proved in \cite{TE3} that
\begin{equation}\label{36.3}
\e_k(\bW^r_q,L_\infty)\asymp k^{-r}(\log k)^{r+1/2}
\end{equation}
holds for $1<q<\infty$, $r>1$. We note that the upper bound in (\ref{36.3}) was proved under condition $r>1$ and the lower bound in (\ref{36.3}) was proved under condition $r>1/q$. Belinskii \cite{Bel} proved the upper bound in (\ref{36.3}) for $1<q<\infty$ under condition $r>\max(1/q,1/2)$.   Relation (\ref{36.3}) for $q=\infty$ under assumption $r>1/2$ was proved in  \cite{TE4}. 

The case $q=1$, $1\le p\le \infty$ was settled by Kashin and Temlyakov \cite{KaTe03}. The authors proved that 
\begin{equation}\label{36.3'} 
\e_k(\bW^r_{1},L_p)\asymp k^{-r}(\log k)^{r+1/2}
\end{equation}
holds for $1\le p<\infty$, $r>\max(1/2,1-1/p)$ and for even $r\in \bbN$
\begin{equation}\label{36.4} 
\e_k(\bW^r_{1},L_\infty)\asymp k^{-r}(\log k)^{r+1}.
\end{equation}

Lemma \ref{LL1} allows us to get from relations (\ref{36.3}) -- (\ref{36.4}) some lower bounds, which are stronger than the ones contained in (\ref{L3}). Lemma \ref{LL1} and  
(\ref{36.3}) imply
 \be\label{LB1}
 k^{-2r}(\log k)^{2r+1/2} \ll \sup_{K\in \bH^{r,4}_{1,\infty}} \e_k(\bW^K_q,L_\infty),\quad 1<q<\infty,\quad r>1/q  
 \ee
and
\be\label{LB2}
 k^{-2r}(\log k)^{2r+1/2} \ll \sup_{K\in \bH^{r,4}_{1,\infty}} \e_k(\bW^K_\infty,L_\infty),\quad   r>1/2.  
 \ee
 Lemma \ref{LL1} and  (\ref{36.3'}) imply for $1\le p<\infty$, $ r>\max(1/2,1-1/p)$
 \be\label{LB3}
 k^{-2r}(\log k)^{2r+1/2} \ll \sup_{K\in \bH^{r,4}_{1,\infty}} \e_k(\bW^K_1,L_p)  .  
 \ee
 Finally, Lemma \ref{LL1} and  (\ref{36.4}) imply that in the case of even $r\in \bbN$ we have
 \be\label{LB4}
 k^{-2r}(\log k)^{2r+1} \ll \sup_{K\in \bH^{r,4}_{1,\infty}} \e_k(\bW^K_1,L_\infty).  
 \ee

 We now give a comparison of the new results on the entropy numbers with the known results on the Kolmogorov widths. Comparing Theorem \ref{aT4.1} with (\ref{36.1}), we see that the behavior of the Kolmogorov widths and the entropy numbers is very different. For instance, in the case $2\le q,p\le \infty$ they have the same rate of decay but in the case 
 $1\le q <2$, $q<p\le \infty$ the entropy numbers are much smaller than the Kolmogorov widths. 
 There are several general results, which give 
upper estimates on the entropy numbers $\e_k(F,X)$ in terms of the Kolmogorov widths $d_n(F,X)$   (see \cite{C}, \cite{VTbook}, p.169, Theorem 3.23, and \cite{VTbookMA}, p.328, Section 7.4). Carl's 
inequality states: For any $r>0$ we have
\begin{equation}\label{D1}
\max_{1\le k \le n} k^r \e_k(F,X) \le C(r) \max _{1\le m \le n} m^r d_{m-1}(F,X).
\end{equation}

We discuss the case $d=1$ and $2\le q,p\le \infty$. We assume that $r$ is large enough for Theorem \ref{InT1} to hold. Then, taking into account that in the case $2\le q,p\le \infty$ we have $\xi(q,p)=0$ we obtain from Theorem \ref{InT1} and inequality (\ref{D1}) the following upper bound
\begin{equation}\label{D2}
\sup_{K \in \bH^{r,2}_{1,1}} \e_k(\bW^K_q,L_p) \ll k^{-2r}.
\ee
Clearly, $\bH^{r,2}_{1,\infty} \subset \bH^{r,2}_{1,1}$. Therefore, the upper bound (\ref{D2}) and the lower bound 
(\ref{L2}) imply the following result.

\begin{Theorem}\label{DT1} Let $d=1$, $2\le q,p\le \infty$, and $1\le w \le \infty$. Then for $r>3/2$ we have
\begin{equation}\label{D3}
\sup_{K \in \bH^{r,2}_{1,w}} \e_k(\bW^K_q,L_p) \asymp k^{-2r}.
\ee
\end{Theorem}

Note, that the upper bound in Theorem \ref{DT1}, which was derived from Theorem \ref{InT1},  gives a slightly better result than in (\ref{L2}). The proof of Theorem \ref{InT1} is based on the connection between the Kolmogorov widths of classes $\bW^K_q$ and the best bilinear approximations of the kernel $K$. Then deep results on the bilinear approximations are used. We explain that important connection in one simple case. For $d\in \bbN$ consider the dictionary
\be\label{Bi6}
 \Pi(d) := \{u(\bx)v(\by)\,:\, \|u\|_{L_2(\bbT^d)}=\|v\|_{L_2(\bbT^d)}=1\},
\ee
where the parameter $d$ in $ \Pi(d)$  indicates that the functions $u$ and $v$ are functions of $d$ variables. 
Define the best $m$-term bilinear approximation of $K(\bx,\by)$, $\bx =(x_1,\dots,x_d)$, $\by =(y_1,\dots,y_d)$, in the norm $L_\bp$ as follows
$$
\sigma_m(K,\Pi(d))_\bp := \inf_{c_j, g_j\in \Pi(d)}\left\|K(\bx,\by) -\sum_{j=1}^m c_jg_j(\bx,\by)\right\|_{L_\bp}.
$$

 \begin{Remark}\label{BIR0} Note that in a number of papers the following notation is used for the best $m$-term bilinear approximation
 \be\label{tau}
 \tau_m(f)_\bp := \sigma_m(f,\Pi(d))_\bp.
 \ee
 \end{Remark}

It is easy to check (see \cite{VT26}) that for a continuous function $K(\bx,\by)$ one has 
$$
d_n(\bW^K_1,L_p) = \tau_n(K)_{p,\infty}.
$$

{\bf Conclusions.} The known result -- Theorem \ref{InT1} -- gives the right order of the characteristic $d_n(\bH^{r,2}_{1,1},L_q,L_p)$ (actually, $d_n(\bH^{r,2}_{1,w},L_q,L_p)$, $1\le w\le \infty)$. The main results of our paper -- relations (\ref{L2}) and (\ref{L3}) -- give  upper bounds, which are close to the lower bounds, for the characteristics $\e_k(\bH^{r,2}_{1,\infty},L_q,L_p)$ and $\e_k(\bH^{r,4}_{1,\infty},L_q,L_p)$ for all $1\le q,p\le\infty$. In the case $2\le q,p\le\infty$ Theorem \ref{DT1} is stronger than (\ref{L2}) in the following sense. (1) It provides the right upper bound, which coincides (in the sense of order) with the lower bound. (2) It gives the upper bound for the wider class $\bH^{r,2}_{1,1}$ than the class $\bH^{r,2}_{1,\infty}$ in (\ref{L2}). However, Theorem \ref{MT1} is much stronger than the corresponding bounds obtained from Theorem \ref{InT1} by the Carl's inequality (\ref{D1}). 
In this regards we formulate three specific open problems.

{\bf Open problem 1. Conjecture.} For large enough $r$ prove
 \be\label{OP1}
 \sup_{K\in \bH^{r,2}_{1,\infty}} \e_k(\bW^K_1,L_\infty) \ll k^{-2r}.
 \ee
 
 {\bf Open problem 2.} Could we replace in Theorem \ref{MT1} (see (\ref{Md1})) the class $\bH^{r,2}_{1,\infty}$ by 
 a larger class $\bH^{r,2}_{1,1}$? 
 
 {\bf Open problem 3.} Assume that Conjecture in Open problem 1 is correct. Could we replace in (\ref{OP1})   the class $\bH^{r,2}_{1,\infty}$ by a larger class $\bH^{r,2}_{1,1}$? 
 
 Note, that in the proof of Theorem \ref{MT1} we used deep known results on the entropy numbers and only added 
 relatively elementary new arguments. As we pointed out above, the proof of Theorem \ref{InT1} is based on the connection between the Kolmogorov widths of classes $\bW^K_q$ and the best bilinear approximations of the kernel $K$ and on deep results on the bilinear approximations. 
 Certainly, it would be interesting to prove analogs of Theorem \ref{InT1} for $d>1$ and Theorem \ref{MT1} for $d>2$ with as sharp logarithmic factors as possible. Also, it would be interesting to study other asymptotic characteristics, for instance, errors of optimal numerical integration and errors of optimal sampling recovery.  
 
  {\bf Acknowledgements.}  
This work was supported by the Russian Science Foundation under grant no. 23-71-30001, https://rscf.ru/project/23-71-30001/, and performed at Lomonosov Moscow State University.

 \Addresses
 
\end{document}